\documentclass[11pt]{amsart}
\usepackage{graphicx,color}
\usepackage[left=1in,right=1in,top=1in,bottom=1in]{geometry}
\usepackage{latexsym}
\usepackage{cite}
\usepackage{amsmath,amssymb,amsthm,mathrsfs}
\newcommand{\N}{\mathbb{N}}

\newtheorem*{cor}{Corollary}

\allowdisplaybreaks[1]  

\newtheorem{theorem}{Theorem}
\newtheorem{proposition}[theorem]{Proposition}

\newtheorem{lemma}[theorem]{Lemma}
\newtheorem{definition}[theorem]{Definition}

\newcommand{\tmax}{t_{\mathrm{max}}}

\newcommand{\wb}{\widetilde{B}}

\newcommand{\wphi}{\widetilde{\phi}}

\renewcommand{\phi}{\varphi}

\let\epsilon=\varepsilon
\newcommand{\R}{\mathbb R}

\def\crn#1#2{{\vcenter{\vbox{
\hbox{\kern#2pt \vrule width.#2pt height#1pt
 }
\hrule height.#2pt}}}}

\newcounter{mnotecount}[section]
\let\oldmarginpar\marginpar
\renewcommand\marginpar[1]{\-\oldmarginpar[\raggedleft\footnotesize #1]%
{\raggedright\footnotesize #1}}

\begin{document}

\title[Homogeneous RG-2 Flow]{Second-Order Renormalization Group Flow of Three-Dimensional Homogeneous Geometries}

\author{Karsten Gimre}
\address[Gimre]{Department of Mathematics, Columbia University, New York City, New York}
\email{gimre@math.columbia.edu}
\author{Christine Guenther}
\address[Guenther]{Department of Mathematics and Computer Science, Pacific University,
Forest Grove, Oregon, 97116}
\email{guenther@pacificu.edu}
\author{James Isenberg}
\address[Isenberg]{Department of Mathematics, University of Oregon, Eugene, Oregon}
\email{isenberg@uoregon.edu}

\date{\today}

\keywords{RG flow; homogeneous geometries.}
\subjclass[2000]{Primary  53C44; Secondary  35K55}
\thanks{Research supported in part by NSF grant PHY-0968612 at Oregon and DGE-1144155 at Columbia.}

\begin{abstract}
We study the behavior of the second order Renormalization Group flow on locally homogeneous metrics on closed three-manifolds. In the cases $\mathbb R^3$  and $\text{SO}(3)\times \R$, the flow is qualitatively the same as the Ricci flow. In the cases $\text{H}(3)$ and $\text{H}(2)\times \R$, if
the curvature is small, then the flow expands as in the Ricci flow case, while if  the curvature is large,
then the flow contracts and forms a singularity in finite time.  The main focus of the paper is the flow on the $\text{SU}(2)$, $\text{Nil}$, $\text{Sol}$, and $\text{SL}(2,\R)$  3-geometries, with two of the three principal directions set equal. The configuration spaces for these geometries are two dimensional, and we can consequently apply phase plane techniques to the study. For the $\text{SU}(2)$ case, the flow is everywhere qualitatively the same as Ricci flow. For the $\text{Nil}$,  $\text{Sol}$, and $\text{SL}(2,\R)$  cases, we show that the configuration space is partitioned into two regions which are delineated by a solution curve of the flow that depends on the coupling parameter: in one of the regions, the flow develops cigar or pancake singularities characteristic of the Ricci flow, while in the other both directions shrink. In the $\text{Nil}$ case we obtain a characterization of the full 3-dimensional flow. 
\end{abstract}

\maketitle

\section{Introduction}\label{intro}

The Ricci flow for a family of metrics $g$ on a manifold $M^n$ is well-known to be the first-order approximation to the Renormalization Group (RG) flow corresponding to perturbative analyses of nonlinear sigma model quantum field theories from a world sheet into $(M^n,g)$ \cite{FriedPRL, FriedAP, Carf}:
\begin{equation*}
\partial_t g_{ij} = -\alpha R_{ij} -\frac{\alpha^2}{2} R_{iklm} R_j^{klm} + \mathcal O(\alpha^3).
\end{equation*}
 Here we use the parameter $\alpha$ to denote the (positive) coupling constant for such quantum field theories. If we carry out appropriate rescalings of the deformation parameter $t$ for the RG flow, then the PDE generating the second-order approximation (in $\alpha$) to the RG flow can be written as
\begin{equation}
\label{RG-2eqn}
\partial_t g_{ij} = -2R_{ij} -\frac{\alpha}{2} R_{iklm} R_j^{klm},
\end{equation}
where $R_{ij}$ and $R^i_{klm}$ are the Ricci and the Riemann curvature tensors corresponding to the (evolving) metric $g_{ij}(t)$, and indices are lowered and raised using $g_{ij}(t)$ and its inverse $g^{kl}(t)$.

There is no consensus among researchers concerning whether it is physically useful to consider the second-order terms in the RG flow while ignoring the influence of higher order terms (which involve cubic and higher order products of the curvature). Whether or not this turns out to be the case, the flow equation (\ref{RG-2eqn}) is mathematically interesting as a (non-linear) deformation of the Ricci flow (whose governing equation is obtained from (\ref{RG-2eqn}) by setting $\alpha=0$),
and accordingly some  of the mathematical features of its flow have been studied in recent years \cite{O, GO}.

In this work, we continue the study of the flow generated by (\ref{RG-2eqn}) (which we label as the ``RG-2 flow equation"), focusing on the following issue: If we fix a family of geometries (preserved by both the Ricci flow and the RG-2 flow) and fix a value of the parameter $\alpha$, does the RG-2 flow have asymptotic behavior similar to that of  the Ricci flow? Further, how does the asymptotic behavior of the RG-2 flow depend on $\alpha?$ 

To study this issue, we have chosen to work with sets of geometries for which the behavior of Ricci flow is well understood: families of 3-dimensional geometries which are locally homogeneous. Since we are not concerned here with the topology of $M^3$, and since every locally homogeneous geometry $(M^n,g)$ lifts to a homogeneous geometry on the universal cover of $M^n$, we assume that the geometries $(M^3,g)$ of interest are all homogeneous.\footnote{A Riemannian geometry $(M, g)$ is defined to be \emph{locally homogeneous} if, for every pair of points $p,q\in M$, there exist neighborhoods $U_p$ of $p$ and $V_q$ of $q$ such that there is an isometry $\Psi_{pq}$ mapping $(U_p, g|_{U_p})$ to $(V_q,g|_{V_q})$ with $\Psi_{pq}(p) = q$. Generally, these local isometries do not extend to isometries of the whole space $(M,g)$. If they do, then the geometry is \emph{homogeneous}, which means that for every pair of points  $p,q\in M$ there is exists an isometry $\Phi_{pq}$ of  $(M,g)$ which maps $p$ to $q$. In this case, the isometry group of the geometry acts transitively on $M$.}
The full range of 3-dimensional homogeneous geometries, and the behavior of the Ricci flow on these geometries, is discussed elsewhere
(see \cite{IJ, KM, GP}\footnote{Note that in some of these references, volume-normalized Ricci flow rather than standard Ricci flow is considered.}). Throughout most of this paper we are concerned in particular with the comparing the two flows on the $\text{SU}(2)$, $\text{Nil}$, $\text{Sol}$, and $\text{SL}(2,\R)$ families of locally homogeneous 3-geometries with two of the three principal directions set equal.  While this extra  condition (which is known as \emph{local rotational symmetry}, or ``\emph{LRS}"  to those who study spatially homogeneous relativistic cosmologies \cite{RS})  is not always preserved by the RG-2 flow, in the cases that we study here, it is. 

A naive comparison of the RG-2 equation (\ref{RG-2eqn}) with the Ricci flow equation $\partial_t g_{ij} = -2R_{ij}$ suggests that the second term in equation (\ref{RG-2eqn}) becomes important -- and can lead to differences in asymptotic behavior of the two flows -- if (roughly speaking) the product of $\alpha$ times the curvature is comparable to unity  either initially or at some time along the flow.  Our analysis below of the flows for the $\text{Nil}$,  $\text{Sol}$,  $\text{SL}(2,\R)$, and $\text{H}(3)$ families of geometries bears this out. On the other hand, we find that for the $\R^3, \text{SO}(3)\times \R$ and  $\text{SU}(2)$ cases, the magnitude of $\alpha \times$curvature does not affect the qualitative asymptotic behavior of RG-2 flow, since, as seen below, in these cases both terms in equation (\ref{RG-2eqn}) have the same sign for all geometries.

For the $\text{SU}(2)$ family of homogeneous LRS 3-geometries, we find that for both the RG-2 flow and the the Ricci flow, if we start at any initial geometry, then all three directions eventually contract  (``shrinker asymptotics") and approach isotropy, with a curvature singularity reached in finite time.  We note that for both of these flows, the phase plane  divides into two regions: If we denote by $(A,B,C)$ the diagonal components of the metric in the Milnor frame, and if we choose $B=C$ to impose the LRS condition,
then the two regions are divided by the solution trajectory $A=B$.

For $\text{Nil}$,  the Ricci flow for all initial geometries is ``immortal" (non-singular for all future time) and is characterized by two expanding directions and one shrinking direction (``pancake asymptotics"). In the RG-2 flow case, however, again setting $B=C$, we find that the phase plane is partitioned into two regions of differing behavior, this time with the boundary given by the curve $A=\frac{2}{3\alpha}B^2.$ For initial geometries with $A_0<\frac{2}{3\alpha}B_0^2$, we have $\tmax=\infty$ and $A(t) \rightarrow 0, B(t)\rightarrow \infty$ as $t \rightarrow \tmax$ (here $\tmax$ is defined so  that $[0,\tmax)$ is the maximal interval of existence of the flow); this behavior matches that of the Ricci flow. For initial geometries with $A_0\geq\frac{2}{3\alpha}B_0^2$,  one has $\tmax <\infty$ and $A(t), B(t) \rightarrow 0$ as $t \rightarrow \tmax.$ For the $\text{Nil}$ family of geometries, unlike the others, we find that 
the full 3-dimensional system of ODEs (with no LRS condition imposed) can be reduced to the LRS case, and so we have the following complete description of the RG-2 flow for all $\text{Nil}$ geometries: Given $A_0,$ $B_0$, and $C_0$, if $\alpha \ge \frac{2B_0 C_0}{3A_0}$ then $\tmax < \infty, $ and $A(t), B(t), C(t) \rightarrow 0$ as $t \rightarrow \tmax$. If, on the other hand,  $\alpha < \frac{2 B_0 C_0}{3A_0}$, then $\tmax = \infty$ and $A(t) \rightarrow 0,$ while $B(t),C(t)\rightarrow \infty$ as $t\rightarrow \infty$. 

 The Ricci flow for all (3-dimensional, LRS) $\text{Sol}$ initial geometries is immortal and is characterized by  ``cigar asymptotics" (one expanding direction and two shrinking directions for volume-normalized flow; one expanding and two unchanging directions for unnormalized flow). For the RG-2 flow with the LRS condition $A=C$, the phase plane is partitioned into two regions by the line $B=2\alpha$. For initial geometries with $B_0\ge 2\alpha$ cigar  asymptotics develop with $\tmax  = \infty$, while for those such that $B_0<2\alpha$, shrinker asymptotics occur, with $\tmax<\infty$. 

 For $\text{SL}(2,\R)$,  the behavior is similar to the $\text{Nil}$ case: The Ricci flow for all (LRS) initial geometries is immortal and is characterized by pancake asymptotics. For the RG-2 flow, this time with $B=A$, the phase plane is partitioned into two regions by a solution curve $\phi$ that converges to $C=0$, $ A = 2\alpha.$ Solutions with initial conditions $(C_0, A_0)$ that satisfy $A_0 > \phi(C_0)$ develop pancake asymptotics with $\tmax=\infty$, while those with $A_0 \le \phi(C_0)$ either develop shrinker asymptotics with $\tmax < \infty$, or converge to $(0,2\alpha)$.

What happens to the dichotomous behavior of the RG-2 flow on the $\text{Nil}$, $\text{Sol}$,  $\text{SL}(2,\R)$, and $\text{H}(3)$  geometries if one varies the parameter $\alpha$? In accord with the assessment that the magnitude of  $\alpha \times$curvature plays a key role in determining if RG-2 flow differs qualitatively from Ricci flow, we find that varying $\alpha$ simply results in a uniform shift of the boundary between those initial geometries whose RG-2 flow is like Ricci flow (with pancake or cigar asymptotics), and those whose RG-2 flow has shrinking asymptotics (unlike Ricci flow). We emphasize that for every positive value of $\alpha$, no matter how small, the same dichotomous behavior is found to occur. 

The proofs of our results here rely strongly on the tools of phase plane analysis for 2-dimensional dynamical systems. These tools are applicable because, first of all, if either Ricci flow or RG-2 flow is restricted to a family of locally homogeneous geometries, the flow PDEs become systems of ordinary differential equations (see Section \ref{Analysis}). Further, for three-dimensional geometries, both flows preserve the diagonality of the metric (in the Milnor frame), so without loss of generality one may reduce the number of dynamical variables from six to three. For Ricci flow, volume-normalized flow is easily implemented, further reducing the number of free functions to two \cite{IJ}. Volume normalization is not, however, readily implemented for RG-2 flow, so in our studies here we generally add an extra restriction: that two of the metric coefficients are initially equal. In the cases studied, this LRS condition is preserved by both flows. Hence for these cases, both flows may be treated as two-dimensional dynamical systems.

Before carrying out the analyses of the flows for these families of geometries, we briefly review some of the results obtained in earlier studies of RG-2 flow (see Section \ref{Prev}), and also (in Sections \ref{ConstCurv} and \ref{DirProd}) examine the behavior of RG-2 flow for constant curvature geometries and for geometries which are the direct product of $\R$ with two-dimensional constant curvature geometries.
We next, in Section \ref{LocHom}, write out the RG-2 equations for locally homogeneous 3-geometries generally. Finally, in Section \ref{Analysis} we state and prove our main results.

We note that there is some overlap of our study  with the largely numerical work of Das, Prabhu, and Kar \cite{Das}.

\section{Previous Results for RG-2 Flow}
\label{Prev}

The primary focus of
 this paper is on ways in which the RG-2 flow differs from the Ricci flow. Concerning this difference, one of the 
striking features of the RG-2 flow equation (\ref{RG-2eqn}) is that if one writes it out as a PDE system for the metric coefficients (with respect to a coordinate basis),  one finds that it is generally not parabolic, even if one adds a DeTurck-type diffeomorphism term (effectively choosing a coordinate gauge). Hence, in contrast to the situation for Ricci flow, the RG-2 initial value problem on closed manifolds (or otherwise) is not generally well-posed. One of the simpler manifestations of this feature is seen if one considers the RG-2 flow on a 2-dimensional manifold $M^2$, in which case the flow preserves the conformal class of the initial metric\footnote{In two dimensions, since the curvature tensor satisfies the identity $R_{ijkl}=\frac{R}{2} (g_{il}g_{jk}-g_{ik}g_{jl})$ the RG-2 flow takes the form $\partial_t g_{ij}=-Rg_{ij} -\frac{\alpha}{4}R^2g_{ij}$; it immediately follows that the RG-2 flow preserves the conformal class  of the metric.} ,
and one can write the flow equation as a PDE for the conformal factor, $e^u$ (the metric on $M^2$ is chosen to be $e^u \tilde{g}$, with $\tilde{g}$ fixed). If, further (following Oliynyk \cite{O}), one writes the equation for the conformal factor in linearized form, setting $u = u_* +v$, one has 
\begin{equation}
\label{2dimlin}
\partial_t v = e^{-u_*}\left(1 + \frac{\alpha}{2} R_*\right) \tilde{\Delta} v + F_* v,
\end{equation}
where $R_*$ is the scalar curvature of the metric $e^{u_*}\tilde{g}$ about which the linearization is being done,  $\tilde{\Delta}$ is the Laplacian of the metric $\tilde{g}$, and $F_*$ is some function depending on $\tilde g$ and $u_*$. Clearly from (\ref{2dimlin}), one sees that if the evolving metric $e^{u_*}\tilde{g}$ is such that $(1 + \frac{\alpha}{2} R_*) >0$, then the flow equation is (at that moment) parabolic, while otherwise the flow is not.

Another aspect of RG-2 flow which has been studied is the stability of flat solutions under this flow. Using techniques (i.e., maximal regularity theory) similar to those used to prove the stability of flat solutions on the torus (in any dimension) under Ricci flow, Guenther and Oliynyk \cite{GO} have proven that these same solutions are stable under RG-2 flow. They also prove stability of constant negative curvature geometries under a flow related to RG-2 by the addition of terms which generate homothetic rescalings and diffeomorphisms of the geometries.

\section{RG-2 Flow for Constant Curvature Geometries}
\label{ConstCurv}

Geometrically, the simplest class of locally homogeneous geometries consists of  those with constant curvature; i.e., those for which the Riemann curvature tensor satisfies the condition 
\begin{equation*}
R_{ijkl} = K (g_{il}g_{jk}-g_{ik}g_{jl})
\end{equation*}
for some constant  $K$.  Letting $Rm^2_{ij} :=  R_{iklm} R_j^{klm}$ denote the quadratic curvature term in the RG-2 flow equation, one readily verifies that the constant curvature condition implies (for dimension $n$) 
\begin{align*}
Rm_{ij}^{2} &=2K^{2}(n-1)g_{ij}, \\
R_{ij} &=K(n-1)g_{ij},
\end{align*}
from which it follows that if $g(t)$ evolves via the RG-2 flow equation (\ref{RG-2eqn}) and if $g(0)=g_0$ has constant curvature, then $g(t)$ preserves its conformal class, and we may write $g(t)=\phi(t)g_0$. Since RG-2 flow preserves isometries, we may also presume that $\phi(t)$ is a spatial constant. Noting that $R_{ij}[\phi g]=R_{ij}[g]$ and $Rm^2_{ij}[\phi g]= \frac{1}{\phi}Rm^2_{ij}[g]$, we find that the evolution equation for $\phi$ corresponding to the RG-2 flow of constant curvature geometries is 
\begin{equation}
\label{phieqn}
\partial_t \phi (t)=-2K(n-1)-\frac{\alpha }{\phi (t)}K^{2}(n-1).
\end{equation}
               
In the Ricci flow case ($\alpha=0$), one easily integrates equation (\ref{phieqn}) to obtain (with $\phi(0)=1$)
\begin{equation*}
g(t) =(1-2K(n-1)t)g_0.
\end{equation*}

With $\alpha$ nonzero, equation (\ref{phieqn}) is more difficult to integrate; however, one does obtain the following implicit solution for $\phi(t)$:
\begin{equation}
\label{RGConstcurv}
\phi (t) =-2K(n-1)t+1 
+\frac{\alpha K}{2}\ln \left\vert \frac{2\phi (t)+\alpha K}{
2+\alpha K}\right\vert .
\end{equation}

In the case of positive curvature $K$, we see from the evolution equation (\ref{phieqn})  that for every value of $K$ and of $\alpha$, $\phi$ monotonically decreases, and we see from the implicit solution (\ref{RGConstcurv}) that $\phi(T)=0$ for 
\begin{equation*}
T=\frac{1}{2K(n-1)}+\frac{\alpha }{4(n-1)}\ln \left\vert \frac{\alpha
K}{2+\alpha K}\right\vert .
\end{equation*}
Hence for positive constant curvature, the RG-2 flow is qualitatively the same as the Ricci flow. We note that since (for $\alpha>0$)
\begin{equation*}
\frac{\alpha K}{2+\alpha K}<1,
\end{equation*}
in fact RG-2 flow shortens the time to reach the singularity.

For negative curvature $K$, the asymptotic behavior of RG-2 flow \emph{does} depend on the values of $K$ and $\alpha$. Indeed, inspecting equation (\ref{phieqn}) at $t=0$ with  $\phi(0)=1$, one finds that for a fixed value of $\alpha$,  those constant negative curvature geometries with $|K|<\frac{2}{\alpha}$ initially \emph{expand} under RG-2 flow (as with Ricci flow); but if $|K|>\frac{2}{\alpha}$, the RG-2 flow initially \emph{contracts}. Moreover, those geometries which initially expand continue to do so, and those which initially contract continue to do so as well.  Furthermore the geometries which expand are immortal, while those which contract collapse in finite time.  We note that for any choice of constant negative curvature $K$, there are choices of $\alpha$ for which the RG-2 flow behaves like Ricci flow, and others for which it behaves very differently.

Since all metrics on $\text{H}(3)$ are constant curvature metrics, the above calculations give a complete description of the RG-2 flow in this case. 

For $\R^3$ the geometries are all flat,  so all are clearly fixed points of the RG-2 flow. 

\section{RG-2 Flow for $\text{SO}(3)\times \R$, and $\text{H}(2) \times \R$.}
\label{DirProd}

While the homogeneous simple direct product geometries on $\text{SO}(3)\times \R$, and $\text{H}(2) \times \R$ do not have constant curvature, the behavior of their RG-2 flows follows immediately from the calculations done above in Section \ref{ConstCurv}. 
In the case of $\text{SO}(3)\times \R$, the metrics take the product form $$g=Dg_{\R} + E\gamma_{S^2},$$ where $D$ and $E$ are spatial constants, $g_{\R}$ is the metric on $\R$, and $\gamma_{S^2}$ is the round metric on the sphere. The calculations of Section \ref{ConstCurv} imply that under the RG-2 flow, the round sphere shrinks to a point. The metric is flat in the $\R$ direction, and therefore as for the Ricci flow (discussed in \cite{IJ}),  a curvature singularity forms in finite time. 

For $\text{H}(2) \times \R$, we also have a product metric of the form $$g=D g_{\R} + E\gamma_{\text{H}(2)},$$ where $\gamma_{\text{H}(2)}$ is the constant curvature metric on the hyperbolic plane, and $D$ and $E$ are spatial constants. Again it follows from the calculations in  the above section that if $|K|>\frac{2}{\alpha}$ the flow contracts in the $\text{H}(2)$ direction, developing a singularity in finite time, while if $|K|<\frac{2}{\alpha}$ then the flow expands for all time. 

\section{The RG-2 Flow Equations for $\text{SU}(2)$, $\text{Nil},$ $\text{Sol},$ and $\text{SL}(2,\R)$}
\label{LocHom}

In this section we focus on four families: those characterized by the (unimodular) transitive isometry groups $\text{SU}(2)$, $\text{Nil}$, $\text{Sol}$, and $\text{SL}(2,\R)$. As geometrically defined flows, both Ricci flow and RG-2 flow preserve isometries, and hence preserve these families. 

There are two different ways to study geometric flows on families of homogeneous geometries. In both approaches, one works with a frame field that is left-invariant under the group action. Since the groups being considered here are unimodular, one may choose the frame field so that the metric is diagonal. Such a choice is useful since, in terms of such a frame field basis, one verifies that both terms on the right  hand side of equation (\ref{RG-2eqn}) are diagonal; thus diagonality is preserved by the RG-2 flow. In the first of the two approaches, one fixes the chosen frame field $\{f_i\}_{i=1}^3$. Hence the frame field commutators $[f_i,f_j]={c}^k_{ij}f_k$ do not evolve in time, but the diagonal components of the metric--which we label $A, B,$ and $C$--do evolve. In the second approach, one allows the frame field, and therefore the commutators, to evolve, but one requires the frame field to be orthonormal; hence the metric coefficients do not evolve.  While both methods are useful\footnote{Both methods are also used in the study of the dynamics of spatially homogeneous cosmologies satisfying the Einstein gravitational field equations.}, here we use the first, which facilitates comparison with the Ricci flow results in \cite{IJ}.

For convenience of reference, we use the notation of \cite{CLN}. Thus we set $\lambda:={c}^1_{23}, \mu:={c}^2_{31}$, and $\nu:={c}^3_{12}$.  The frame field $\{e_i\}_{i=1}^3$ defined by $e_1 = A^{-1/2} f_1, e_2 = B^{-1/2}f_2$, and $e_3 = C^{-1/2} f_3$ is orthonormal. Calculating the curvatures for geometries specified by $\{ \lambda, \mu, \nu\}$ (which identify the isometry family) and by $\{ A, B, C\}$ (which specify the metric in that family), one obtains sectional curvatures (see e.g., section 7 of chapter 4 in \cite{CLN})
\begin{align*}
&K(e_{2}\wedge e_{3}) =\frac{(\mu B-\nu C)^{2}}{4ABC}+\lambda \frac{2\mu
B+2\nu C-3\lambda A}{4BC}, \\
&K(e_{3}\wedge e_{1}) =\frac{(\nu C-\lambda A)^{2}}{4ABC}+\mu \frac{2\nu
C+2\lambda A-3\mu B}{4AC}, \\
&K(e_{1}\wedge e_{2}) =\frac{(\lambda A-\mu B)^{2}}{4ABC}+\nu \frac{%
2\lambda A+2\mu B-3\nu C}{4AB}.
\intertext{In the original frame field $\{f_i\}$,  one computes from this}
&K_{12}=K(f_1\wedge f_2) = (AB)K(e_1 \wedge e_2), \\
&K_{23}=K(f_2\wedge f_3)=(BC)K(e_2 \wedge e_3), \\
&K_{31}=K(f_3\wedge f_1)=(AC)K(e_3\wedge e_1).
\end{align*}

We can easily calculate the Ricci curvatures using the orthonormal frame field, then converting to the $\{f_i\}$ frame field, and thereby obtaining
\begin{align*}
&R_{11} =Rc(f_1\wedge f_1)=\frac{(\lambda A)^{2}-(\mu B-\nu C)^{2}}{2BC}, \\
&R_{22} =Rc(f_2\wedge f_2)=\frac{(\mu B)^{2}-(\nu C-\lambda A)^{2}}{2AC}, \\
&R_{33}=Rc(f_3\wedge f_3) =\frac{(\nu C)^{2}-(\lambda A-\mu B)^{2}}{2AB}.
\intertext{ Similarly, we find}
&Rm^2_{11} = Rm^2(f_1\wedge f_1)=\frac{2}{AB^2} K^2_{12}+\frac{2}{AC^2}K_{31}^2, \\
&Rm^2_{22}=Rm^2(f_2\wedge f_2)=\frac{2}{A^2B}K_{12}^2+\frac{2}{C^2B}K_{23}^2,\\
&Rm^2_{33}=Rm^2(f_3\wedge f_3)=\frac{2}{A^2C}K_{31}^2+\frac{2}{B^2C}K_{23}^2.
\end{align*}

Substituting these calculations into equation (\ref{RG-2eqn}), we obtain an ODE system for the evolution of the metric coefficients $\{ A, B, C\}$ under RG-2 flow: 
\begin{align}
\label{generalODEsA}
\frac{d A}{d t} &=\frac{(\mu B-\nu C)^{2}-(\lambda A)^{2}}{BC}\\
\nonumber&\quad- \alpha A\left[\left(\frac{(\lambda A-\mu B)^{2}}{4ABC}+\nu \frac{%
2\lambda A+2\mu B-3\nu C}{4AB}\right)^2 +\left(\frac{(\nu C-\lambda A)^{2}}{4ABC}+\mu \frac{2\nu
C+2\lambda A-3\mu B}{4AC}\right)^2\right],\\
\label{generalODEsB}
\frac{d B}{d t} &= \frac{(\nu C-\lambda A)^{2}-(\mu B)^{2}}{AC} \\
\nonumber&\quad-\alpha B\left[\left(\frac{(\lambda A-\mu B)^{2}}{4ABC}+\nu \frac{%
2\lambda A+2\mu B-3\nu C}{4AB}\right)^2+\left(\frac{(\mu B-\nu C)^{2}}{4ABC}+\lambda \frac{2\mu
B+2\nu C-3\lambda A}{4BC}\right)^2\right],\\
\label{generalODEsC}
\frac{d C}{d t} &=\frac{(\lambda A-\mu B)^{2}-(\nu C)^{2}}{AB}\\
\nonumber&\quad-\alpha C\left[\left(\frac{(\nu C-\lambda A)^{2}}{4ABC}+\mu \frac{2\nu
C+2\lambda A-3\mu B}{4AC}\right)^2
+\left(\frac{(\mu B-\nu C)^{2}}{4ABC}+\lambda \frac{2\mu
B+2\nu C-3\lambda A}{4BC}\right)^2\right].
\end{align}

As noted above, our analysis here relies on imposing an LRS condition, which sets two of the three metric coefficients $\{A,B,C\}$ equal. Such a condition leads to a useful reduction of the analysis only if the evolution equations preserve that equality. Below, for each of the four isometry families we consider, we do find that there is an LRS equality which is preserved.

\section{Analysis of RG-2 Flow for Four Families of Locally Homogeneous Geometries}
\label{Analysis}

In this section, we consider four families of locally homogeneous geometries--$\text{SU}(2)$, $\text{Nil}$, $\text{Sol}$, and $\text{SL}(2,\R)$ --each of them identified by a particular choice of the constants $\{ \lambda, \mu, \nu\}$.  For each family, we first write out the evolution equations, we next impose an LRS condition and write out the resulting reduced ODE system, and finally we state and prove a theorem which characterizes the asymptotic behavior of the RG-2 flow for that family. 

As we mentioned above and as we show below, for $\text{SU}(2)$ we find that the asymptotic behavior of the RG-2 flow is essentially the same as that of the Ricci flow, for all initial geometries. For each of the other cases, there is a region in the phase plane in which the two flows behave in a qualitatively similar way, and a region in which the RG-2 flow and the Ricci flow behave  qualitatively very differently.

Before stating these results more precisely and proving them, we note some general features of the phase plane analysis which we use here for all four families of geometries. For each of these families, the evolution equations take the form of a system of two ODEs 
\begin{align}
\label{M}
\frac{dM}{dt}&=F(M,N)\\
\label{N}
\frac{dN}{dt}&=G(M,N)
\end{align} 
to be solved for $M(t)>0$ and $N(t)>0$, with $F(M,N)$ and $G(M,N)$ a pair of specified rational functions. Two special features of these equations (true for all four families) strongly restrict the allowed behavior of the solutions. First, the denominators of the rational functions $F(M,N)$  and $G(M,N)$ are simple monomials in the metric functions $M$ and $N$. Hence, so long as $M$ and $N$ are positive, the right hand sides of (\ref{M}) and (\ref{N}) are well-behaved, and consequently, the solutions continue (for all positive values of $t$) so long as $M$ and $N$ stay bounded and non-zero. Second, in all cases, we find that one or the other of these evolution equations--let us say, without loss of generality, the first--has a negative definite right hand side (for positive $M$ and $N$). This has the important consequence that there are no equilibrium solutions, and further that all bounded solution trajectories must approach one or the other (or both) of the axes. This property also allows us to replace the evolution parameter $t$ by $M$ (keeping the reversed direction of the flow in mind), and then work with the orbit flow equation 
\begin{equation}
\label{orbitflow}
\frac{dN}{dM}=\frac{F(M,N)}{G(M,N)},
\end{equation}
obtained from the above system. Keeping in mind the above-noted properties, we derive (below) the salient features of the orbits of the flow in the $(M>0,N>0)$ (quarter) phase plane for each family of geometries by studying the behavior of solutions of \eqref{orbitflow}, emphasizing their behavior as $M$ decreases (and $t$ increases).
To determine whether the solutions are immortal, we return to the ODE system (\ref{M})-(\ref{N}).

\subsection{$\text{SU}(2)$ geometries}

For the homogeneous geometries with $\text{SU}(2)$ symmetry, one has $\lambda = \mu = \nu = -2$, so then the RG-2  evolution equations (\ref{generalODEsA})-(\ref{generalODEsC}) take the form 
\begin{align*}
\frac{d A}{d t} &=\frac{4(C-B)^2-4A^2 }{BC}\\
\nonumber&\quad- \frac{\alpha}{AB^2}\left[\frac{(B-A)^2}{C}+2A+2B-3C\right]^2 -\frac{\alpha}{AC^2}\left[ \frac{(A-C)^2}{B}+ 2C+2A-3B\right]^2,\\
\frac{d B}{d t} &=\frac{4(A-C)^2-4B^2}{AC} \\
\nonumber&\quad- \frac{\alpha}{A^2B} \left[\frac{(B-A)^2}{C}+2A+2B-3C\right]^2-\frac{\alpha}{C^2B} \left[\frac{(C-B)^2}{A}+2B+2C-3A\right]^2, \\
\frac{d C}{d t} &=\frac{4(B-A)^2-4C^2}{AB}\\
\nonumber&\quad-\frac{\alpha}{A^2C}\left[\frac{(A-C)^2}{B}+ 2C+2A-3B\right]^2-\frac{\alpha}{B^2C}\left[\frac{(C-B)^2}{A}+2B+2C-3A\right]^2.
\end{align*}

One readily verifies in these equations that if one sets $A=B$ in the first two of them, then $\frac{d}{dt} A= \frac{d}{dt}B$; hence the LRS condition $A=B$ is preserved by the RG-2 flow. Similarly, for these geometries, LRS conditions $B=C$ and $A=C$ are also preserved by the RG-2 flow. We choose here (without loss of generality) to set $B=C$; consequently we work with the reduced system 
\begin{align}
\label{SU2LRS:A}
\frac{dA}{dt}&=-\frac{4A^2B^2+2\alpha A^3}{B^4},\\
\label{SU2LRS:B}
\frac{dB}{dt}&=\frac{4AB^2-8B^3-10\alpha A^2-16\alpha B^2+24\alpha AB}{B^3}.
\end{align}

It is immediate from \eqref{SU2LRS:A} that for all non-negative values of the coupling constant $\alpha$ (including $\alpha =0$, which corresponds to Ricci flow) and for all values of the evolving metric coefficients $A$ and $B$, the metric coefficient $A$ decreases monotonically in time. The behavior of $B$ (and equivalently $C$) is not so immediately apparent. 

Since our goal is to show that for both Ricci flow and RG-2 flow, in fact $B$ does eventually become a monotonically decreasing function which asymptotically approaches $A$, it is useful to calculate the evolution of the quantity $A-B$; we obtain
\begin{equation}
\label{SU2:A-B}
\frac{d(A-B)}{dt} = -\frac{4A^2}{B^2} - \frac{2\alpha A^3}{B^4} + 8- \frac{4A}{B} + 2\alpha\left[\frac{5}{B}\left(\frac{A}{B}-1\right)^2 + \left(\frac{3}{B} - \frac{2A}{B^2}\right)\right].
\end{equation}
Setting $A=B$ in \eqref{SU2:A-B}, we find that the right hand side vanishes\footnote{This vanishing does \emph{not} immediately follow from fact that, if we set $A=B$ in the original $SU(2)$ system above, we obtain 
$\frac{d}{dt} A= \frac{d}{dt}B$, since we are working now with the $(B=C)$ reduced LRS system \eqref{SU2LRS:A}-\eqref{SU2LRS:B}.}.
It follows that for both Ricci flow and RG-2 flow, full isotropy $A=B=C$ is preserved. From the point of view of our $(A,B)$-parametrized phase portrait, this tells us that the diagonal line $A=B$ consists of orbits of the flows. Noting that, for $A=B$, the equation (\ref{SU2LRS:A}) takes the form $\frac{dA}{dt}=-4-\frac{\alpha}{A}$, we see that all of these solutions approach the $A=0=B$ origin in finite time.

This feature of the $A=B$ diagonal line plays a crucial role in our phase portrait analysis of the RG-2 flow and Ricci flow for LRS $\text{SU}(2)$ geometries. Since the line $A=B$ corresponds to solution curves that partition the plane, orbits on a given side of the $A=B$ diagonal at any given time must stay on that side for all time, as a consequence of the  uniqueness of solutions. Hence, if we define $\mu:=B-A$, the phase portrait analyses for $\mu>0$ and for $\mu<0$ can be carried out completely independently. 

Before examining the RG-2 flow for the LRS $\text{SU}(2)$ geometries, we focus on the Ricci flow for these geometries. Starting with either $\mu>0$ or $\mu<0$, we show that along every flow line, this quantity 
approaches zero, signaling that $A$ and $B$ approach each other, and the geometries approach isotropy. 

Following our treatment of the Ricci flow for LRS $\text{SU}(2)$ geometries, we prove essentially the same results for the RG-2 flow for these geometries.

\subsubsection{Ricci Flow case} 
As noted above, for all values of $\alpha$ and for all of these geometries $(A>0$ and $B>0)$, one has $\frac{dA}{dt}<0$. Hence, for the purposes of phase portrait study, we can replace the parameter $t$ by $A$, and work with $B(A)$, or equivalently $\mu(A)$. In carrying out these studies, it is important to keep in mind that increasing $t$ corresponds to decreasing $A$. Therefore, asymptotic (future) Ricci flow behavior is studied by examining orbits with decreasing values of $A$.

For the Ricci flow case ($\alpha=0$), we calculate from equation \eqref{SU2:A-B} the following evolution equation
\begin{equation*}
\frac{d\mu}{dA} = \frac{\mu(3A+2\mu)}{A^2},
\end{equation*}
which has the explicit (general) solution 
\begin{equation}
\label{SU2RicSoln}
\mu(A) = \frac{-A^3}{A^2 - k},
\end{equation}
where $k$ is any constant. We see from \eqref{SU2RicSoln} that if for some value $A=A_0$ one has $\mu(A_0)>0$, then we must have $k>A_0^2$. It then follows that as $A$ decreases toward zero, $\mu(A)$ stays positive but decreases to zero, with $\lim_{A \rightarrow 0} \mu(A) =0.$

If, on the other hand, $\mu(A_0)<0$, then we must have $k<A_0^2$. Recalling the definition $\mu:=B-A$ and the requirement that both $B$ and $A$ be positive, we must have $-A_0 < \mu(A_0) < 0$, from which it follows that $k<0$ in expression \eqref{SU2RicSoln}. We then have 
\begin{equation*}
\mu(A)=-A\frac{A^2}{A^2+(-k)}>-A
\end{equation*} 
for all $A<A_0$. It follows in this case that as $A$ decreases toward zero, $\mu(A)$ stays negative but increases to zero, with (again) $\lim_{A \rightarrow 0} \mu(A) =0.$

Once it has been determined (as above) that the orbits of the Ricci flow on the $(A>0,B>0)$ plane (for LRS $\text{SU}(2)$ geometries) all proceed to the point $(0,0)$, it remains to show that the solutions of the system (\ref{SU2LRS:A})-(\ref{SU2LRS:B}) (with $\alpha = 0$) do indeed all approach this point (without any prior singularities halting the flow). This follows immediately as an application of the general statement made above: so long as $A$ and $B$ are positive and finite, the flow continues to $(0,0)$.

We may in fact show that the solutions all reach $(0,0)$ (and become singular) in finite time. The key to showing this is the Ricci flow equation 
\begin{equation}
\label{RFSUA}
\frac{dA}{dt}=-4\frac{A^2}{B^2}. 
\end{equation}
For those solutions initially (and therefore always) below the $A=B$ line, it follows from (\ref{RFSUA}) that $\frac{dA}{dt}<-4$, so $A\rightarrow 0$ in finite time. For those solutions above the $A=B$ line, it is useful to calculate 
\begin{equation*}
\frac{d}{dt}\bigg(\frac{A}{B}\bigg) = 8\frac{A}{B^2}\bigg(1- \frac{A}{B}\bigg); 
\end{equation*}
combining this with the presumption that $\frac{A}{B}<1$, we have $\frac{d}{dt}(\frac{A}{B}) > 0$. Hence $\frac{dA}{dt} < -4 (\frac{A_0}{B_0})^2$, which again implies that a singularity is reached, with $A\rightarrow 0$, in finite time.

\subsubsection{RG-2 case}
\label{RG2SU(2)}

In this case we have
$$\frac{d\mu}{dt} = -4\mu\bigg(\frac{3A+2\mu}{(A+\mu)^2}\bigg) - 2\alpha \mu \frac{(5A^2 + 12A\mu + 8\mu^2)}{(A+\mu)^4},$$
and we calculate

\begin{align*}
\frac{d\mu}{dA} &= \frac{2\mu(A+\mu)^2(3A+2\mu)+\mu \alpha(5A^2+12A\mu+8\mu^2)}{2A^2(A+\mu)^2+\alpha A^3}\\
 &=\frac{6\mu A^3+16A^2\mu^2+5 \alpha A^2\mu+14A\mu^3+12 \alpha A\mu^2+8 \alpha \mu^3+4\mu^4}{2A^4+4A^3\mu+2A^2\mu^2+A^3}.
\end{align*}

We see immediately from this equation that if $\mu>0$, then $\frac{d\mu}{dA} >0$. Hence, in this case, as $A$ decreases towards zero, $\mu$ decreases as well. It follows from this, together with our determination above that all bounded trajectories must approach the axes, that any solution trajectory with $\mu>0$ must approach a point on the $A=0$ axis, with a finite value of $B$. We seek (below) to show that in fact all of these trajectories approach the origin $(0,0)$. 

If, on the other hand, a solution trajectory has $\mu<0$, then we infer from the discussion above that the trajectory must approach the $B=0$ axis. In this case as well, we show that the trajectories all approach the origin.

We focus first on the trajectories below the $A=B$ line: those with $\mu<0$. We presume, for the sake of contradiction, that there is a trajectory which approaches  the $B=0$ axis at some finite $A=A_1>0$. To show that this presumption leads to a contradiction, it is not useful to examine the system (\ref{SU2LRS:A})-(\ref{SU2LRS:B}) directly, since while the right hand sides of both equations blow up as $(A,B) \rightarrow (A_1,0)$, such behavior is in principle consistent. Rather, we note that (since $\frac{dA}{dt}<0)$ the trajectory of a solution of (\ref{SU2LRS:A})-(\ref{SU2LRS:B}) must everywhere satisfy the \emph{trajectory ODE} 
\begin{eqnarray}
\label{SU2dB/dA}
\frac{dB}{dA} = \frac{4B^4 -2AB^3 +\alpha(5A^2B-12AB^2+8B^3)}{2A^2B^2 +\alpha A^3}. 
\end{eqnarray}
Further, if there were a solution trajectory which approached $(A_1,0)$, since the right hand side of (\ref{SU2dB/dA}) is well-behaved at and near $(A_1,0)$, then it would follow from standard ODE theory that indeed this trajectory (everywhere along its path satisfying \eqref{SU2dB/dA}) must intersect and pass through $(A_1,0)$. 

The contradiction arises because ODE theory guarantees that the solutions of (\ref{SU2dB/dA}) in a neighborhood of the point $(A_1,0)$ are unique, and we readily verify that $B(A)=0$ is a solution. Hence there are no solutions intersecting the $B=0$ axis uniquely at $(A_1,0)$. This implies that there are no solutions of the system (\ref{SU2LRS:A})-(\ref{SU2LRS:B}) which approach $(A_1,0)$.
We note that this argument breaks down at the origin, since the right hand side of equation (\ref{SU2dB/dA}) is \emph{not} well-behaved at the origin. 

We can make a similar argument for the portion of the phase plane that is above the line $A=B$, with $\mu>0$.  Again, the issue is to show that the solutions in this region, which we have determined are bounded with decreasing $\mu$, do not asymptotically approach the $A=0$ axis, except at the origin. To argue this, we first note from equation (\ref{SU2LRS:B}) that for $B>A$,  $B(t)$ monotonically decreases. Hence in a neighborhood of the $A=0$ axis, the trajectory of a solution of (\ref{SU2LRS:A})-(\ref{SU2LRS:B}) may be studied as a function $A(B)$,  satisfying the trajectory ODE
\begin{eqnarray}
\label{SU2dA/dB}
\frac{dA}{dB}=\frac{A^2(2B^2+\alpha A)}{B(4B^3+8\alpha B^2-2AB^2-12\alpha AB+5\alpha A^2)}.
\end{eqnarray}
We now assume that 
there is a solution trajectory which approaches $(B_1,0)$, with $B_1>0$. Then as argued above, since the right hand side of (\ref{SU2dA/dB}) is well-behaved in a neighborhood  $(B_1,0)$, such a solution passes through this point, and is the only one which does so. However, 
$A(B)=0$ is also a solution of this ODE which passes through $(B_1,0)$, leading to a contradiction. We have thus determined that all of the RG-2 solutions for locally rotationally symmetric $\text{SU}(2)$ geometries approach the origin. 

Finally, we argue that these solutions reach the origin in finite time. As for the Ricci flow solutions, we rely on the equation for the metric coefficient $A$--equation (\ref{SU2LRS:A}), for the general RG-2 case. If $A>B$, then it follows easily from (\ref{SU2LRS:A}) that $\frac{dA}{dt}<-4$, so $A\rightarrow 0$ in finite time. If $A<B$, then since 
\begin{equation*}
\frac{d}{dt}\bigg(\frac{A}{B}\bigg) = 8\frac{A}{B^2}\bigg(1- \frac{A}{B}\bigg) +8 \alpha \frac{A}{B^3}(1-\frac{A}{B})(2-\frac{A}{B}) 
\end{equation*}
we have $\frac{A}{B}$ increasing. Therefore $\frac{dA}{dt} < -4 (\frac{A_0}{B_0})^2$, and we have $A$ collapsing to zero in finite time. 

Combining these results with those established above, we have proven the following:

\begin{theorem} [RG-2 Flow for Locally Rotationally Symmetric $\text{SU}(2)$ Geometries]
Every solution of the system (\ref{SU2LRS:A})-(\ref{SU2LRS:B}) becomes singular in finite time, with ($A(t), B(t))$ approaching $(0,0)$ at the singularity.
\end{theorem}
It follows from this theorem that for these geometries, RG-2 flow is qualitatively the same as Ricci flow, with all solutions having shrinker asymptotics, and with all solutions approaching isotropy.

\subsection{$\text{Nil}$ geometries}
\label{Nilsubsection}

For the homogeneous $\text{Nil}$ geometries, we have $\lambda=-2,$ and $\mu=\nu=0$; hence  the ODEs (\ref{generalODEsA}), (\ref{generalODEsB}), and  (\ref{generalODEsC}) take the form

\begin{align}
\frac{dA}{dt}&=\frac{-4A^2}{BC} - \frac{\alpha}{AB^2} (\frac{A^2}{C})^2 - \frac{\alpha}{AC^2}(\frac{A^2}{B})^2\label{dAdtnil0},\\
\frac{dB}{dt}&=\frac{4A}{C}-10\alpha\frac{A^2}{BC^2} \label{dBdtnil0},\\
\frac{dC}{dt}&=\frac{4A}{B}-10\alpha\frac{A^2}{B^2C}\label{dCdtnil0}.
\end{align}

As in the case of $SU(2),$ one easily verifies that if one sets $B=C$ in the first two equations, then $\frac{d}{dt} B= \frac{d}{dt}C$; it follows that  the LRS condition $B=C$ is preserved by the RG-2 flow. We now set $B=C$ and work with the reduced system 
\begin{align}
\frac{dA}{dt}&=-\frac{4A^2}{B^2}-2\alpha\frac{A^3}{B^4}\label{dAdtnil}\\
\frac{dB}{dt}&=\frac{4A}{B}-10\alpha\frac{A^2}{B^3}.\label{dBdtnil}
\end{align}
We note that for the $\text{Nil}$ geometries (unlike the $\text{SU}(2)$ geometries), we cannot choose $A=C$ or $A=B$ as LRS conditions; $B=C$ is the only one that works. 

As for the $\text{SU}(2)$ geometries, regardless of the (non-negative) value of $\alpha$,  the right hand side of the evolution equation (\ref{dAdtnil}) is negative definite. Hence we can carry out much of our study of the RG-2 flow for $\text{Nil}$ geometries working with trajectory functions $B(A)$, which satisfy the trajectory equation
\begin{equation}
\label{mainnil}
\frac{dB}{dA}=\frac{B(5\alpha A-2B^2)}{A(2B^2+\alpha A)}.
\end{equation}
We start by considering the Ricci flow ($\alpha=0$) case.

\subsubsection{Ricci Flow Case}

For $\alpha=0$, the system of ODEs takes the simple form
\begin{align*}
\frac{dA}{dt}&=-\frac{4A^2}{B^2}\\
\frac{dB}{dt}&=\frac{4A}{B},
\end{align*}
and we readily verify that for any initial data $(A_0, B_0)$, this system has the explicit solution
\begin{align*}
A(t) &= \frac{k_1}{12}(k_1t+k_2)^{-1/3} \\
B(t)&=(k_1t+k_2)^{1/3},
\end{align*}
where $k_1 := \frac{(12A_0)^2}{B_0}$ and $k_2 = (\frac{B_0}{12A_0})^3.$ 
One can see immediately that $A(t) \rightarrow 0$ and $B(t) \rightarrow \infty$ as $t \rightarrow \tmax=\infty$. Recalling that we have presumed that
 $B=C$, we see that all of these LRS $\text{Nil}$ solutions are immortal, and all have pancake asymptotics.

In fact, one can argue that \emph{all} Ricci flow solutions for $\text{Nil}$ (LRS or not) have this same behavior. To see this, we first note that the general $\text{Nil}$ ODE system (\ref{dAdtnil0})-(\ref{dCdtnil0}) implies that   $\frac{d}{dt}(\frac{B}{C}) = 0$. Consequently, for any solution with initial data $(A_0,B_0,C_0)$, we have $C(t)=\frac{C_0}{B_0} B(t)$, which allows the general system (\ref{dAdtnil0})-(\ref{dCdtnil0}) to be essentially reduced to the LRS system 
(\ref{dAdtnil})-(\ref{dBdtnil}). The LRS Ricci flow behavior thus holds for the general Ricci flow behavior for $\text{Nil}$ geometries.

\subsubsection{RG-2 Case}
To analyze the behavior of the RG-2 flow, we start by seeking explicit solutions of the trajectory equation (\ref{mainnil}). Motivated by the form of the right hand side of (\ref{mainnil}), we find that 
\begin{equation*}
B=\sqrt{\frac{3 \alpha }{2}}A^{\frac{1}{2}}
\end{equation*} 
is indeed a solution. Substituting this relation into the evolution equation (\ref{dAdtnil}), we obtain $\frac{dA}{dt}=-\frac {32}{9\alpha} A$. It follows that the solutions of the system (\ref{dAdtnil})-(\ref{dBdtnil}) lying on this parabolic orbit (which we label 
$\pi_{\text{Nil}}$  are immortal, decaying exponentially to the origin $(0,0)$.

The trajectory $\pi_{\text{Nil}}$ partitions the phase plane, and as a consequence of the well-posedness of the ODE initial value problem associated to (\ref{mainnil}) for positive $A$ and $B$, it cannot be crossed by any other trajectory. As we see below, those RG-2 flow solutions lying above $\pi_{\text{Nil}}$ (with $B_0>\sqrt{\frac{3 \alpha }{2}}A_0^{\frac{1}{2}}$) behave much like the Ricci flow solutions, while those lying below $\pi_{\text{Nil}}$ (with $B_0<\sqrt{\frac{3 \alpha }{2}}A_0^{\frac{1}{2}}$) behave very differently.

The LRS $\text{Nil}$ solutions below  $\pi_{\text{Nil}}$ in fact behave to an extent like the LRS $\text{SU}(2)$ solutions below the $A=B$ line, and the arguments  to show this are  similar: We first note that any solution $B=\phi(A)$ of the trajectory equation (\ref{mainnil}) which lies below  $\pi_{\text{Nil}}$ satisfies the condition $5\alpha A-2[\phi(A)]^2>0$, from which it follows that $\frac{d\phi}{dA}>0.$ Thus $B(t)$, along with $A(t)$, monotonically decreases along a solution with a trajectory below $\pi_{\text{Nil}}$. Earlier considerations guarantee that these solutions continue so long as $B$ is positive, and the regularity of the right hand side of (\ref{mainnil}) for $B>0$ implies (as argued for the $\text{SU}(2)$ geometries) that they approach the origin rather than a point on the $B=0$ axis. Hence, geometrically, these solutions all exhibit shrinker asymptotics.

We would like to show that these solutions (below $\pi_{\text{Nil}}$) become singular in finite time. To do this, we work with (\ref{dAdtnil}), the evolution equation for $A(t)$, which for convenience we write in the form $\frac{dA}{dt}=\xi$, with $\xi:=-\frac{4A^2}{B^2} - 2\alpha \frac{A^3}{B^4}$. Noting that $\xi$ is negative for all solutions, we see that if we can show that $\xi$ is a decreasing  quantity along any solution below $\pi_{\text{Nil}}$ then it follows that, for such a solution, $\frac{dA}{dt}\le -\xi_0$, where $\xi_0$ is calculated from initial data. Finite time singularities for these solutions would then follow. Calculating the time derivative of $\xi$, we obtain
\begin{equation}
\label{dxidt}
\frac{d\xi}{dt}=\frac{8A^3}{B^8}(8B^4 - \alpha B^2 A - 34 \alpha^2 A^2).
\end{equation}
Noting that the only positive zeroes for the right hand side of (\ref{dxidt}) are given by $A=\frac{8}{17\alpha} B^2$, we readily determine that indeed, for solutions below $\pi_{\text{Nil}}$ (i.e., those with $\frac{A}{B^2} >\frac{2}{3\alpha}$), $\frac{d\xi}{dt}$ is negative. It follows that $(A,B)$ reaches $(0,0)$ in finite time

We proceed now to consider those solutions \emph{above} $\pi_{\text{Nil}}$. The partition of the phase plane requires that these solutions (all of which have decreasing $A(t)$) either (i) approach the origin; (ii) approach the $A=0$ axis at a finite value of $B$, say $B_1$; (iii) approach the $A=0$ axis with $B \rightarrow \infty$; or (iv) approach $B \rightarrow \infty$ at a finite value of $A$, say $A_1$. 

To rule out the first possibility, we consider a trajectory $B=\phi(A)$ which satisfies (\ref{mainnil}) and lies above $\pi_{\text{Nil}}$ and compare it to the trajectory$\pi_{\text{Nil}}$ itself, which satisfies $B=\pi_{\text{Nil}}(A) =\sqrt{\frac{3 \alpha }{2}}A^{\frac{1}{2}}$.
Writing the right hand side of (\ref{mainnil}) abstractly as a function  $f:(0,\infty)\times(0,\infty)\to\mathbb{R}$ which takes the form 
\begin{equation*}
f(x,y)=\frac{y(5\alpha x-2y^2)}{x(\alpha x+2y^2)},
\end{equation*}
we calculate 
\begin{equation*}
\frac{\partial f}{\partial y}(x,y)=\frac{5\alpha^2x^2-16\alpha xy^2-4y^4}{x(\alpha x+2y^2)^2}
\end{equation*}
and determine that if $\frac{x}{y^2}>\frac{3 \alpha}{2}$, then $\frac{\partial f}{\partial y}(x,y)<0$. This implies that $\frac{d\phi}{dA}(A)<\frac{d\pi_{\text{Nil}}}{dA}(A)$ for all $A$ in the domain of $\phi$, from which it follows that $\lim_{A\rightarrow 0} \phi(A)$ cannot equal zero. 

To rule out possibility (iv), we suppose for the purpose of contradiction that there exists $A_1>0$ such that $\phi(A)$ approaches $\infty$ as $A \rightarrow A_1$. To work  effectively with infinite values of $B$, it is useful to  define $\wb:=e^{-1/B}$ and $\wphi:=e^{-1/\phi}$, for which the corresponding trajectory ODE is 
\begin{equation}
\label{dBtilde/dAnil}
\frac{d\wb}{dA}=\frac{\wb\ln \wb(2-5\alpha A(\ln \wb)^2)}{A(2+\alpha A(\ln \wb)^2)}.
\end{equation}
In terms of these variables, we have $\wphi(A)\to 1$ as $A\rightarrow A_1$. Observing that the trajectory ODE (\ref{dBtilde/dAnil}) is well behaved for all positive values of $A$ and $\wb$, we see that the solution $\wphi(A)$ can be extended to a neighborhood of $A_1$. But this leads to a contradiction, since $\hat \phi(A)=1$ is also a solution of  (\ref{dBtilde/dAnil}) which  passes through the point $(A_1,1)$, and is not equal to $\wphi(A).$

With possibility (iv) ruled out, we have established that for all solutions, $A(t)$ approaches zero. We now argue that in all cases (above $\pi_{\text{Nil}}$), $B(t)$ approaches infinity. We first note, from (\ref{mainnil}), that the parabola  $B=\sqrt{\frac{5 \alpha }{2}}A^{\frac{1}{2}}$ partitions the phase plane into a region (above this parabola, and along the $A=0$ axis) in which $\frac{dB}{dA}$ is negative, and a region in which this quantity is negative. Both regions intersect the region of solutions we are now considering--those above $\pi_{\text{Nil}}$. In view of the form of these regions, and since we know (with (i) ruled out) that these solutions do not approach the origin, we see that in all cases, for $A$ sufficiently close to 0, we must have $B=\phi(A)$ satisfying  $\frac{dB}{dA}<0.$ Thus (recalling that $A(t)$ monotonically decreases) we have $B$ monotonically increasing with time, and increasing as $A$ approaches zero. It therefore must reach a limit, finite or infinite.

To show that in fact, this limit for $B$ is infinite, we presume for the purpose of contradiction that there exists $B_1<0$ such that $\lim_{A\rightarrow 0} \phi (A) =B_1.$ In view of the monotonicity of $\phi(A)$ near $A=0$, we may invert $\phi$ and consider $\phi^{-1}(B)$ as a solution of 
\begin{equation}
\label{dadb}
\frac{dA}{dB}=\frac{A(2B^2+\alpha A)}{B(5\alpha A-2B^2)}.
\end{equation}
Noting the regularity of the right hand side of this equation in the neighborhood of the point $(0,B_1)$, we use  the now familiar argument based on the well-posedness of the ODE initial value problem for (\ref{dadb}) to reach a contradiction. It follows that $\lim_{A\rightarrow 0} \phi (A) =\infty$ for any solution above $\pi_{\text{Nil}}$. Thus these solutions all have pancake asymptotics, like the Ricci flow solutions. 

To further show that, as for the Ricci flow solutions, these RG-2 solutions with pancake asymptotics are immortal, we again work with the evolution equation for $B(t)$. If it were the case that $B(t)$ approaches infinity in finite time $\tmax $, then it would have to be true that for $t$ sufficiently close to $\tmax $, we have $\frac{d^2 B}{dt^2}$ positive. However, $\frac{d^2 B}{dt^2}$ along a solution is given by the right hand side of (\ref{dxidt}), and we readily determine that for $B$ large and $A$ small, this is always negative. Immortality follows. 

Combining all of these results, we have proven the following:
\begin{theorem} [RG-2 Flow for Locally Rotationally Symmetric $\text{Nil}$ Geometries]
Let $(A(t), B(t))$ be a solution of \eqref{dAdtnil} and \eqref{dBdtnil}. 
\begin{enumerate} 
\item If $\alpha\geq\frac{2}{3}A_0^{-1}B_0^2$ then the solution becomes singular in finite time, with $(A(t), B(t))$  approaching $(0,0)$ at the singularity.
\item If $\alpha<\frac{2}{3}A_0^{-1}B_0^2$ then the solution is immortal, with  $\lim_{t \rightarrow \infty} (A(t),B(t))=(0,\infty)$.
\end{enumerate}
\label{propnil}
\end{theorem}

Just as for Ricci flow for $\text{Nil}$ geometries, we find that the results we have obtained for the RG-2 flow of LRS $\text{Nil}$ geometries holds for $\text{Nil}$ geometries without the LRS condition being imposed. The argument for this is the same as for Ricci flow: Based on equations (\ref{dBdtnil0}) and (\ref{dCdtnil0}), we easily verify that  $\frac{d}{dt}(\frac{B}{C}) = 0$. Thus, for any solution with initial data $(A_0,B_0,C_0)$, we have $C(t)=\frac{C_0}{B_0} B(t)$. If we set $\kappa:=\frac{B_0}{C_0}$, then the system \eqref{dAdtnil0}-\eqref{dCdtnil0} reduces to 
\begin{align*}
\frac{dA}{d\widetilde{t}}&=-\frac{4A^2}{B^2}-\frac{2\alpha\kappa A^3}{B^4},\\
\frac{dB}{d\widetilde{t}}&=\frac{4A}{B}-\frac{10\alpha\kappa A^2}{B^3},
\end{align*}
where we have defined $\widetilde{t} := \kappa t.$
Comparing these equations with \eqref{dAdtnil}-\eqref{dBdtnil}, we see that it 
follows that the analysis done above for the LRS solutions applies to all solutions. We obtain the following corollary:
\begin{cor}
\label{Nil3d}
Let $(A(t), B(t),C(t))$ be a solution of \eqref{dAdtnil0}-\eqref{dCdtnil0}.
\begin{enumerate} 
\item If $\alpha\geq\displaystyle\frac{2B_0C_0}{3A_0}$ then the solution becomes singular in finite time, with $(A(t), B(t),C(t))$  approaching $(0,0,0)$ at the singularity.
\item  If $\alpha<\displaystyle\frac{2B_0C_0}{3A_0}$, then the solution is immortal, with  $\lim_{t \rightarrow \infty} (A(t),B(t),C(t))=(0,\infty,\infty)$.
\end{enumerate}
\end{cor}

\subsection{$\text{Sol}$ Geometries}

In this case we have $\lambda = -2, \mu = 0$ and $\nu=2$, and the ODE system takes the following form: 

 \begin{align*}
\frac{dA}{dt}&=-4\frac{(A^2-C^2)}{BC}-2\alpha\frac{(A+C)^2(A^2-2AC+5C^2)}{AB^2C^2},\\
\frac{dB}{dt}&=4\frac{(A+C)^2}{AC}-2\alpha\frac{(A+C)^2(5A^2-6AC+5C^2)}{A^2BC^2}, \\
\frac{dC}{dt}&=-4\frac{(C^2-A^2)}{AB}-2\alpha\frac{(C+A)^2(C^2-2AC+5A^2)}{A^2B^2C}.
\end{align*}

We again verify  that if we set $A=C$ in the equations above, then $\frac{d}{dt} A= \frac{d}{dt}C$; hence the LRS condition $A=C$ is preserved by the RG-2 flow. 
Setting $A=C$ we obtain the (quite simple) reduced system of equations

\begin{align}
\label{dA/dtSol}
\frac{dA}{dt} &= -16 \alpha \frac{A}{B^2} \\
\label{dB/dtSol}
\frac{dB}{dt} &= 8 - 16 \alpha \frac{1}{B}.
\end{align}
This ODE system is semi-decoupled (in that the second equation (\ref{dB/dtSol}) involves only $B$) and can be solved explicitly if $\alpha=0$, and implicitly if $\alpha>0$. 

\subsubsection{Ricci Flow Case}

Setting $\alpha=0$, we have $\frac{dA}{dt}=0$ and $\frac{dB}{dt} = 8$. Hence the Ricci flow solutions for LRS $\text{Sol}$ geometries take the explicit form  
\begin{align*}
 A &= A_0\\
 B &= 8t +B_0
\end{align*}
for constants $A_0$ and $B_0$. Clearly these solutions exist for all future time, and since the imposed LRS condition sets $C=A$, the solutions all have cigar asymptotics.

\subsubsection{ RG-2 Case:}

It follows from a straightforward phase plane analysis that for the LRS $\text{Sol}$ geometries, the phase plane splits into two regions: In one of these regions ($B\ge 2 \alpha$), RG-2 flow solutions have cigar asymptotics, just like Ricci flow, while in the other ($B\le 2 \alpha$), RG-2 flow solutions have shrinker asymptotics, unlike Ricci flow. We can more easily obtain this result by using implicit  solutions of \eqref{dA/dtSol}-\eqref{dB/dtSol}, as follows.

As noted above (see also \cite{Das}), the LRS $\text{Sol}$ RG-2 evolution equations (\ref{dA/dtSol})-(\ref{dB/dtSol}) admit an implicit general solution, which (for initial data $(A(0), B(0))=(A_0, B_0)$) takes the following form
\begin{equation}
\label{implicitSol1}
A(t)=A_0e^{-4t/\alpha}\qquad\text{and}\qquad B(t)=2\alpha
\end{equation}
or 
\begin{equation}
\label{implicitSol2}
A(t)\left(1-\frac{2\alpha}{B(t)}\right)=A_0\left(1-\frac{2\alpha}{B_0}\right)\quad\text{and}\quad B(t)-B_0+2\alpha\ln\left|\frac{B(t)-2\alpha}{B_0-2\alpha}\right|=8t,
\end{equation}
depending on whether or not $B_0=2\alpha.$ 

 If $B_0=2\alpha$ then it follows from (\ref{implicitSol1}) that the flow is immortal,  with $A(t)$ decreasing to zero and $B(t)$ constant as $t\to\infty.$  
 
Now suppose that  $B_0>2\alpha$. We see from the second part of (\ref{implicitSol2}) that for positive $t$, we must have $B(t) > B_0$. Moreover, translating $t$, we see similarly that for any pair $t_2>t_1$, we must have $B(t_2)>B(t_1)$; hence $B(t)$ monotonically increases.  It then immediately follows from the first part of (\ref{implicitSol2}) that $A(t)$ monotonically decreases. This same equation also tells us that, since its right hand side is constant, $A(t) \to 0$ as $t\to\tmax$ if and only if $1-\frac{2 \alpha }{B(t)} \to \infty$ as $t\to\tmax$. This cannot happen, so $A(t)$ must converge to some $\bar A>0$; we write $\lim_{t \to \tmax } A(t)=\bar A,$ for some $\tmax $ which may or may not be finite.

If we presume that $\tmax $ is finite, then it must be true that $\lim_{t \to \tmax } B(t)=\infty.$
However, this limit is inconsistent with the second part of (\ref{implicitSol2}), so we must in fact have $\tmax =\infty$. It then follows from the second part of (\ref{implicitSol2}) that indeed $\lim_{t \to \infty} B(t)=\infty.$ Combining this with the first part of (\ref{implicitSol2}), we determine that 
\begin{equation*}
\bar{A}=A_0(1-\frac{2\alpha}{B_0}), 
\end{equation*}
 thereby relating the asymptotic value of $A$ to the initial data $(A_0,B_0).$

We suppose now instead that $B_0<2\alpha.$  Just as the second part of (\ref{implicitSol2}) implies that $B(t)$ monotonically increases if $B_0>2$, it implies that $B(t)$ monotonically decreases if $B_0<2\alpha.$ It then follows that there must exist some $\bar B  \in[0,B_0)$ and some $\tmax $ (possibly infinite) such that $\lim_{t \to \tmax } B(t)=\bar B$. Since $\bar B$ is finite, it follows from the second part of (\ref{implicitSol2}) that $\tmax $ is finite as well; so these solutions go singular in finite time. This can happen only if either $A(t)$ or $B(t)$ or both go to zero as $t \to \tmax $. Now if $\bar B >0$, then we must have $\lim_{t \to \infty} A(t)=0$; however this (and  $\bar B >0$) are inconsistent with the first part of (\ref{implicitSol2}). Therefore we must have $\bar B =0$. But then it follows from the first part of (\ref{implicitSol2})
that indeed it must be true that  $\lim_{t \to \infty} A(t)=0$.  We conclude that these solutions go singular in finite time, and have shrinkers asymptotics.

Combining all of these results, we obtain the following theorem:
\begin{theorem}\label{SOL} [RG-2 Flow for Locally Rotationally Symmetric $\text{Sol}$ Geometries]
Let $(A(t), B(t))$ be a solution of (\ref{dA/dtSol}), (\ref{dB/dtSol}) with initial data $(A_0, B_0).$ 
\begin{enumerate}
\item If $B_0>2\alpha,$ then $\tmax=\infty$ and $B(t)\to\infty$ and $A(t)\to A_0(1-\frac{2\alpha}{B_0})$ as $t\to\infty.$
\item If $B_0=2\alpha,$ then $\tmax=\infty$ and $B(t)=2\alpha$ for all $t>0,$ and $A(t)\to 0$ as $t\to\infty.$
\item If $B_0<2\alpha,$ then $\tmax<\infty$ and $A(t),B(t)\to 0$ as $t\to\tmax.$
\end{enumerate}
\end{theorem}

It follows from these results that for the LRS $\text{Sol}$ geometries, the phase plane splits into two regions: In one of these regions ($B\ge 2 \alpha$), RG-2 flow solutions have cigar asymptotics, just like Ricci flow, while in the other ($B\le 2 \alpha$), RG-2 flow solutions have shrinker asymptotics, unlike Ricci flow.

\subsection{$\text{SL}(2,\R)$ Geometries}

As in the $\text{Nil}$ and $\text{Sol}$ cases,  we find that for LRS  $\text{SL}(2,\R)$ geometries the phase plane is partitioned into two regions, one which exhibits behavior similar to the Ricci flow, and one in which the behavior differs. Interestingly, solutions whose initial values start in the non-Ricci flow region in this case can have two distinct behaviors: either $(C(t),A(t)) \to (0,0)$ or $(C(t),A(t)) \to (0,2\alpha)$. Finding the curve that partitions the space is also more involved; instead of directly specifying  a solution of the trajectory equation (\ref{SLRdAdC}), we define a sequence of solutions of (\ref{SLRdAdC}) that converges to a limit solution whose graph partitions the phase plane.

For the $\text{SL}(2,\R)$ geometries, one has $\lambda=\mu=-2,$ and $\nu=2$, so the evolution equations take the form 
\begin{align}
\label{SLA}
\frac{dA}{dt}&=\frac{(2B+2C)^{2}-4A^{2}}{BC}-\frac{\alpha}{B^2C^2A} [2(A^4 + 2A^2(B+C)^2-8A(B-C)(B+C)^2)\\
\nonumber &\quad-2(B+C)^2(5B^2-6BC+5C^2)],\\
\label{SLB}
\frac{dB}{dt}&=\frac{(2C+2A)^{2}-4B^{2}}{AC}-\frac{\alpha}{A^2C^2B} [2(5A^4+4AC(B+C)^2+A^3(-8B+4C)\\
\nonumber &\quad-2(2A^2(B^2-4BC-C^2)+(B+C)^2(B^2-2BC+5C^2))],\\
\label{SLC}
\frac{dC}{dt}&=\frac{(-2A+2B)^{2}-4C^{2}}{AB}-\frac{\alpha}{A^2B^2C} 
[2(5A^4-4A^3(B-2C)-4AB(B+C)^2 \\
\nonumber &\quad-2(-2A^2(B^2+4BC-C^2)+(B+C)^2(5B^2-2BC+C^2)].
\end{align}

If we set $A=B$ in equations (\ref{SLA}) and (\ref{SLB})  we find that $\frac{d}{dt}(A-B)=0$. Thus the LRS condition $A=B$ is preserved for $\text{SL}(2,\R)$ geometries, and we work with the reduced ODE system
\begin{align}
\label{SLRedC}
\frac{dC}{dt}&=\frac{-4C^2}{A^2} -2\alpha \frac{C^3}{A^4},\\
\label{SLRedA}
\frac{dA}{dt}&=8+4\frac{C}{A}-{\alpha}\frac{(16A^2+24AC+10C^2)}{A^3}.
\end{align}

\subsubsection{Ricci Flow Case}

It is immediately clear from this set of equations that if we set $\alpha=0$, then $A(t)$ monotonically increases, while $C(t)$ monotonically decreases. We now show that along every solution, 
$A(t)$ increases to infinity, and there exists a positive constant $\bar{C}$ (generally different from one solution to another) such that $C(t) \to \bar{C}$.  The (reduced) Ricci flow equations take the form 

\begin{align}
\label{SLRFC}
\frac{dC}{dt}&=\frac{-4C^2}{A^2}, \\
\label{SLRFA}
\frac{dA}{dt}&=8+4\frac{C}{A};
\end{align}
consequently we have $\frac{dA}{dt} = 8+4\frac{C}{A}> 8$, from which it follows that   $A(t) \to \infty$. Next, using $(\ref{SLRFA})$ and $(\ref{SLRFC})$ we form the trajectory ODE
$$\frac{dA}{dC} = \frac{-2A^2 - AC}{C^2},$$ 
which has the explicit solution
\begin{equation*}
A(C) = \frac{C}{k C^2 - 1},
\end{equation*}
for some constant $k$. Substituting in the (arbitrary) initial data $(C(0),A(0)) = (C_0, A_0)$, we solve for $k$ and obtain
\begin{equation}
\label{A(C)}
A(C) =\frac{C_0^2 A_0 C }{C^2(A_0 + C_0 ) - C_0^2 A_0}.
\end{equation}
At and near the initial geometry $(C_0, A_0)$, the denominator  $C^2(A_0 + C_0 ) - C_0^2 A_0$ is positive. Indeed, formally inverting \eqref{A(C)}, we find that as $A \to \infty$, $C$ decreases to the (initial data dependent) value $ \bar{C} = (\frac{C_0^2 A_0}{A_0 + C_0})^{1/2}$. Since $B=A$, we see that these flows have pancake asymptotics. We also readily verify that these solutions are immortal.

\subsubsection{RG-2 Case}: 

Choosing $\alpha$ positive, we see immediately from equation (\ref{SLRedC}) that $C(t)$ monotonically decreases for any data, while \eqref{SLRedA} indicates no general monotonicity for $A(t)$. The monotonicity of $C(t)$ allows us to work with trajectories of the form $A(C)$, which satisfy the trajectory ODE
\begin{equation}
\label{SLRdAdC}
\frac{dA}{dC}=-\frac{A(4A^3+2A^2C-8\alpha A^2-12\alpha AC-5\alpha C^2)}{C^2(2A^2+\alpha C)}=:f(A,C).
\end{equation}

Finding the curve which partitions the  LRS  $\text{SL}(2,\R)$ phase plane into a region in which the RG-2 flow has the same asymptotics as the Ricci flow, and a region in which this is not the case, is not as simple as in the Nil and Sol cases. To do it, we specify a sequence of solutions of the trajectory ODE which have their initial values contained on the zero level set of the function $f(A,C)$, and then show that this sequence converges to a curve which solves the trajectory ODE but has its initial point off the level set of $f(A,C)$ (with both numerator and denominator of $f$ approaching zero).  More specifically, we proceed as follows: 1) We 
establish that the zero level set  $\{f(A,C)=0\}$ is the graph of a smooth, strictly increasing function $A=g(C)$. 2) We show that for any solution $\Phi(C)$ of  (\ref{SLRdAdC}), defined on a maximum interval $(C^-_{\Phi}, C^+_{\Phi})$, $\lim_{C \to C^-_{\Phi}} \Phi(C)$ exists (possibly infinite). 3) We specify a sequence of solutions $\phi_n$ of (\ref{SLRdAdC}) and prove that they converge to a solution $\phi$ whose graph partitions the phase plane. After carrying through these three steps, we determine the asymptotics of solutions which lie on either side of $\phi$, verifying that indeed those on one side match the asymptotics of the Ricci flow, and those on the other side do not.

We now carry out the details of these steps. We start with the following result concerning the zero level set of $f(A,C)$:

\begin{lemma}
\label{g(C)}
For $A,C\in(0,\infty)$, the set $\{f(A,C)=0\}$ is the graph of a smooth strictly increasing function $A=g(C)$ which is defined on $(0,\infty)$, and has $\lim_{C \to 0} g(C) =2\alpha$. 
\end{lemma}

\begin{proof}
From (\ref{SLRdAdC}) we see that $f(A,C)=0$ if and only if 
\begin{equation}
\label{QuadraticC}
4A^3 + 2A^2 C - 8\alpha A^2 - 12 \alpha AC - 5 \alpha C^2 = 0.
\end{equation} 
This expression is quadratic in $C$; it is satisfied if and only if  \\ $C=\frac{A}{5\alpha}\left(A-6\alpha\pm\sqrt{(A-6\alpha)^2+20\alpha(A-2\alpha)}\right)$. Only the positive root is consistent with the requirement that $A>0$ and $C>0$, since if we choose the negative root, then $C>0$ implies that we must have 
$$A-6\alpha>\sqrt{(A-6\alpha)^2+20\alpha(A-2\alpha)}> 0,$$
from which it follows that $A>6\alpha>2\alpha$; then 
$$A-6\alpha > \sqrt{(A-6\alpha)^2+20\alpha(A-2\alpha)}> A-6\alpha,$$ which is a contradiction. 
 Consequently $f(A,C)=0$ and $A,C>0$ if and only if
\begin{align}
\label{SL(2R)C=h(A)}
C=\frac{A}{5\alpha}\left(A-6\alpha+\sqrt{(A-6\alpha)^2+20\alpha(A-2\alpha)}\right)=:h(A)
\end{align}

Calculating the derivative of $h(A)$, we obtain
$$h'(A)=\frac{h(A)}{A}+\frac{A}{5\alpha}\left(1+\frac{A+4\alpha}{\sqrt{(A-6\alpha)^2+20\alpha(A-2\alpha)}}\right),$$
which we readily verify is positive so long as $A$ and $h(A)$ are positive. Consequently, it follows from the inverse function theorem that the function $g:=h^{-1}$ exists, and moreover it is smooth and strictly increasing. Noting that the range of $h$ is $(0,\infty)$, we see that the domain of $g(C)$ is $(0,\infty)$ as well. Further, since we readily verify that $h(2\alpha)=0$, we see that the $(0, 2 \alpha)$ is a limit point of the graph of $g(C).$
\end{proof}

The next step is to show that any solution of (\ref{SLRdAdC}) converges as $C$ decreases. The proof of this uses the results we have just established for $g(C).$

\begin{lemma}
\label{PhiC}
Let $\Phi(C)$ be a solution of (\ref{SLRdAdC}), defined on a maximal interval with infimum $C_\Phi^- $. Then $\Phi(C)$ converges to some value (possibly $\infty$) as $C \to C_\Phi^- $. 
\end{lemma}

\begin{proof} 
We first consider solutions $\Phi$ which do not intersect $g$, the graph of the zero level set of $f(A,C)$. Since the domain of $g$ is $(0,\infty)$, the graph of $g$ partitions the phase plane; it follows that $\Phi$ is either bounded above or below $g$, which implies that either $\Phi'>0$ or $\Phi'<0$. In either case, $\Phi$ is monotonic, so convergence as  $C \to C_\Phi^- $ follows.

We now consider solutions which do intersect $g(C)$; so there   
exists $c$ such that $\Phi(c)=g(c).$ We first show that there is at most one such $c$, and then show that the limit as $C\to {C^-_{\Phi}}$  exists. Since (by definition of $g$) $\Phi'(c)=0$ and since (as shown in the proof of Lemma \ref{PhiC}) $g'(c)>0$, we determine that $\Phi(C)>g(C)$ for $C\in(c-\varepsilon,c)$ for some $\varepsilon>0.$ Now  say there exists $\tilde c<c$ such that $\Phi(\tilde c)=g(\tilde c)$, and choose the largest such $\tilde c$. Then since $\Phi'(\tilde c)=0$ and since $g'(\tilde c)>0$ it follows  that $\Phi(C)<g(C)$ for $C\in(\tilde c, \tilde c+\tilde \varepsilon)$ for some $\tilde \varepsilon>0$. Since $\Phi(C) > g(C)$ on $(c-\epsilon,c)$, it follows from  the intermediate value theorem that there exists some $\hat c$ with $\Phi(\hat c) = 0$ and $\tilde c  < \hat c < c$, which contradicts the assumption that $\tilde c$ is the largest such value.  We conclude that such a $\tilde c$ cannot exist. Consequently $\Phi(C)>g(C)$ for all $C<c$, which implies that $\Phi'(C)<0$ for $C<c.$  The existence of $\lim_{C \to C^-_{\Phi}} \Phi$ (possibly infinite) then follows.

\end{proof}

We now define the sequence $\phi_n$ of solutions (designed to converge to the partitioning solution $\phi$) by specifying their initial data as a sequence of points along the graph of $g$:

\begin{definition} 
\label{phin}For each $n\in\N$, $\varphi_n$ is the solution of \eqref{SLRdAdC} such that $$\varphi_n\left(\frac{1}{n}\right)=g\left(\frac{1}{n}\right).$$  
\end{definition}

It is clear from \eqref{SLRdAdC} that the solution $\phi_n$ exists on some interval surrounding $\frac{1}{n}$.  We wish to show that in fact, the maximal domain of $\phi_n$ is $(C_n^-, \infty)$, with $0 \le C_n^- < \frac{1}{n}$, and with $\lim_{C \to C_n^-} \phi_n = \infty$.
We do this in two steps, via the following two lemmas:

\begin{lemma}
\label{phin}
The maximal domain of the  function $\varphi_n$  (as defined in Definition \ref{phin})  includes $[\frac{1}{n},\infty)$.
\end{lemma}

\begin{proof} Labeling the maximal domain of $\varphi_n$ as $(C_n^-,C_n^+)$, for $\frac{1}{n} \in (C_n^-,C_n^+)$, our task in proving this lemma is  to show that $C_n^+ = \infty$. To show this, we argue by contradiction. 
Supposing that $C_n^+< \infty$, we have either (i) $\lim_{C \to C_n^+} \phi_n(C)=0$, or (ii) $\lim_{C \to C_n^+} \phi_n(C)=\infty$.

To rule out the first possibility, we use the familiar argument that since the right hand side of the trajectory ODE \eqref{SLRdAdC} is well-behaved in the neighborhood of any hypothesized limit point $(C_n^+, 0)$, and since $A=0$ is a solution of \eqref{SLRdAdC} which passes through that point, then it follows from the well-posedness of the initial value problem for \eqref{SLRdAdC} near $(C_n^+, 0)$ that there can be no other solution with this limit point.

We now suppose that (ii) holds. We know from the definition of $\phi_n$ that $\phi_n'(1/n) = 0$ and we know from Lemma \ref{g(C)} that  $g'(1/n) > 0$; hence it follows  that  $\varphi_n(C)<g(C)$ on some interval $(\frac{1}{n},\frac{1}{n}+\varepsilon)$. Since the domain of $g(C)$ is $(0,\infty)$ and since we are assuming that $\phi_n(C) \rightarrow \infty$ as $C \to C_n^+< \infty$,  there must exist a value $\tilde c$ with  $\frac{1}{n}< \tilde c < C_n^+$ such that  $\varphi_n(\tilde c)=g(\tilde c)$. However, as argued in the proof of Lemma \ref{PhiC}, there cannot be two values of $c$ with $\phi_n(c) = g(c)$, so we consequently have a contradiction. This rules out case ii), and we conclude that $C_n^+ = \infty.$ 
\end{proof}

\begin{lemma}
The maximal domain of the function $\varphi_n$ is $(C_n^-, \infty)$ for $0 \le C_n^-< \frac{1}{n}$, with \\ $\lim_{C\to C_n^-} \phi_n(C) =\infty.$
\end{lemma}

\begin{proof}
The tasks here are to show that $C_n^- \ge 0$, and to verify the indicated limit. 

To show that negative values for $C$ are not contained in the maximal (connected) domain of the solution $\phi_n$, we observe that for $0< C< \frac{1}{n}$, one has $\phi_n(C)  > g(C)> 2\alpha$ and $\phi'_n(C) <0$. It follows that $f(\phi_n(C),C)$ is badly behaved for $C \to 0$, and consequently  $C_n^-\ge 0.$

We now verify  that $\varphi_n(C)\to \infty$ as $C\to C_n^- .$  From Lemma \ref{PhiC} we know that $\varphi_n(C)$ does indeed converge to some value (possibly $\infty$) as $C\to C_n^- .$  If $C_n^- >0$ this value must be $\infty,$ since otherwise $f(C_n^-, \phi_n(C_n^-))$ would be well defined and the solution $\phi_n$  could be extended beyond that point. If $C_n^- =0$, the fact that we cannot have $\varphi_n(C)\to \bar{A}<\infty$ as $C\to 0$ follows by our usual argument, relying  in this case on $dC/dA$, which is given by the inverse of the right hand side of \eqref{SLRdAdC} . 
 \end{proof}

With the above results established for the solutions $\phi_n: (C_n^-,\infty) \to \R^+$, we now show that the sequence $(\phi_n)$ converges to a function $\phi:(0,\infty)\to \R+$ which is a solution of the trajectory ODE, and which splits the phase plane for the LRS \text{SL(2,R)} geometries into two regions of differing behavior for the RG-2 flow.

\begin{proposition}
\label{propositionphi}
The sequence of solutions  $(\phi_n)$ converges to a solution $\phi$ of \eqref{SLRdAdC} with domain $(0,\infty)$. Furthermore, given any solution $\Phi$ of (\ref{SLRdAdC}),

i) if  $\Phi(c) > \phi(c)$ for some $c>0$, then there exists $C_\Phi^-  \ge 0$ such that 
$\Phi \to \infty$ as $C \to C_\Phi^- .$

ii) if $\Phi(c) \le \phi(c)$ for some $c > 0$, then $\Phi(C)$ converges to either 0 or $2\alpha$ as $C \to {C_\Phi^- }.$

\end{proposition}

\begin{proof}
It follows from the proof of Lemma \ref{PhiC} that $\varphi_n(C)>g(C)$ for $C<\frac{1}{n}$. In particular, we have $\varphi_n(\frac{1}{n+1})>g(\frac{1}{n+1})=\varphi_{n+1}(\frac{1}{n+1}),$ so that (as a consequence of the uniqueness of solutions to the trajectory ODE)  $\varphi_n>\varphi_{n+1}$ everywhere. Relying on this monotonicity property, we verify that the pointwise limit $\phi(C):=\lim_{n \to \infty} \phi_n(C)$ exists for all $C\in (0, \infty)$.

To show that this function $\phi$ is a solution of \eqref{SLRdAdC}, we consider any pair of points $\hat C >C>0$, and calculate 
$$\varphi_n(\hat C)=\varphi_n(C)+\int_C^{\hat C}f\big(\varphi_n(c),c\big)\ dc,$$
from which it follows that
$$\varphi(\hat C)=\varphi(C)+\lim_{n\to\infty}\int_C^{\hat C}f\big(\varphi_n(c),c\big)\ dc.$$
We verify that the function $f(C,A)$ is bounded on the set 
 $\{(c,a):c\in[C,C'],a\in[g(c),\varphi_{k(C)}(c)]\}$, since $f$ is continuous and since this set is compact (here $k(C)$ is any element of $\N\cap(\frac{1}{C},\infty)$ such that $\varphi_k$ is defined for any $c\in[C,C']$).  We may therefore use  the dominated convergence theorem to show that 
$$\varphi(C')=\varphi(C)+\int_C^{C'}f(\varphi(c),c)\ dc. $$
It follows immediately that $\varphi$ is a solution of \eqref{SLRdAdC}.  Furthermore, since $\varphi_n(\frac{1}{n})=g(\frac{1}{n})$ and since $g(C)\to 2\alpha$ as $C\to 0$, we have $\varphi(C)\to 2\alpha$ as $C\to 0.$

Since the domain of $\phi$ is $(0,\infty)$, the graph of this solution clearly partitions the phase space. Uniqueness of solutions guarantees that solutions cannot cross $\phi$, and therefore must remain either above or below its graph. It readily follows that if a solution $\Phi$ satisfies the inequality $\Phi(c)>\varphi(c)$ for some $c$, then there exists some $n\in\N$ such that $\varphi_n(c)<\Phi(c).$ Then  since $\varphi_n(C)\to\infty$ as $C\to C_n^- $ , it follows that there exists some $C_\Phi^-  \geq C_n^- $ such that $\Phi(C)\to \infty$ as $C\to C_\Phi^- .$ If instead a solution $\Phi$ satisfies the inequality $\Phi(c)\leq\varphi(c)$ for some $c>0$, then one verifies that $\Phi(C)$ converges  either to $0$ or to $2\alpha$ as $C\to 0.$
\end{proof}

While this proposition describes the asymptotic behavior of the two classes of solutions, it says nothing about their life spans. We determine these now. For those solutions $(C(t),A(t))$  of  \eqref{SLRedC} and \eqref{SLRedA} which satisfy the initial data inequality $A_0>\varphi(C_0)$, it follows from Proposition \ref{propositionphi} that $A(t)$ (and $B(t)$ as well) become infinite, while $C(t)$ converges to a constant (pancake asymptotics). Examining the evolution equation for $A$, we have 
$$\frac{dA}{dt} \le 8 + 4\frac{C_0}{A},$$ 
with $A \to \infty$. This implies that these solutions are immortal. 

For the solutions with initial data $(C_0, A_0)$ satisfying the inequality $A_0\leq\varphi(C_0)$, Proposition \ref{propositionphi} together with the usual axis-avoiding arguments shows that these solutions all converge to either the origin, or to the point $(C,A) = (0, 2 \alpha).$ We have not determined the life span of those converging to $(0, 2\alpha)$, but we argue as follows that those converging to the origin do so in finite time. 
Using equations  \eqref{SLRedC} and \eqref{SLRedA}, we calculate 
$$\frac{d}{dt}(\frac{A}{C^2}) = \frac{8}{C^2}(1-\frac{2\alpha}{A}) + \frac{12}{CA}(1-\frac{2\alpha}{A})-\frac{6\alpha}{A^3}.$$
Since $A$ converges to zero,  there exists a time $t_1$ such that $A < 2\alpha$ for all subsequent times. It follows that $\frac{d}{dt}(\frac{A}{C^2}) < 0$ for $t > t_1$, and consequently  $-\frac{C^2}{A}<-\frac{C_1^2}{A_1} $ where $C_1 = C(t_1), $ and $A_1 = A(t_1). $ Assuming also that $A < 1,$ we have $-\frac{C^2}{A^2}<-\frac{C_1^2}{A_1} .$ Using these inequalities,  we determine that (for $t > t_1)$ 
\begin{align}
\frac{dC}{dt} &= -4\frac{C^2}{A^2} - 2\alpha\frac{C^3}{A^4}\\
&< -4\frac{C_1^2}{A_1}.
\end{align}
It follows that these solutions have  a finite extinction time. 

Combining these results concerning solution life span with the results of Proposition \ref{propositionphi}, we obtain the following theorem, which shows that, as with the Nil and Sol geometries, there is a solution which partitions the phase plane into a region in which the RG-2 and the Ricci flow are very similar asymptotically, and another in which they are very different. 
\begin{theorem} [RG-2 Flow for Locally Rotationally Symmetric $\text{SL}(2,\R)$ Geometries]
Let $(A(t),C(t))$ be solutions of \eqref{SLRedC} and \eqref{SLRedA} with $(C_0, A_0)=(C(0), A(0))$, and let $\phi$ be the limit solution determined in Proposition \ref{propositionphi}. 
\begin{enumerate} 
\label{SL2RSummary}
\item If $A_0>\varphi(C_0)$ then the solution is immortal with $A(t)\to\infty,$ and $C(t)\to\overline{C}\geq 0$ as $t\to\infty.$
\item If $A_0\leq\varphi(C_0)$ then either the solution converges to $(0,0)$ in finite time, or the solution converges to $(0,2\alpha).$
\end{enumerate}
\end{theorem}

\end{document}